\documentclass[12pt]{article} 

\usepackage{amsmath}
\usepackage{amsthm}
\usepackage{amsfonts}
\usepackage{mathrsfs}
\usepackage{stmaryrd}
\usepackage{setspace}
\usepackage{fullpage}
\usepackage{amssymb}
\usepackage{breqn}
\usepackage{enumitem}
\usepackage{bbold} 
\usepackage{authblk}
\usepackage{comment}
\usepackage{hyperref}
\usepackage{pgf,tikz}
\usepackage{graphicx}
\usetikzlibrary{decorations.pathreplacing,arrows}
\tikzstyle{vertex}=[circle,draw=black,fill=black,inner sep=0,minimum size=3pt,text=white,font=\footnotesize]

\newtheorem{thm}{Theorem}[section]

\newtheorem{lemma}[thm]{Lemma}

\newtheorem{proposition}[thm]{Proposition}

\newtheorem{corollary}[thm]{Corollary}

\newtheorem{clm}[thm]{Claim}
\newtheorem{observation}[thm]{Observation}
\newtheorem*{thm*}{Theorem}

\newcommand\ex{\ensuremath{\mathrm{ex}}}
\newcommand\cA{{\mathcal A}}
\newcommand\cB{{\mathcal B}}

\newcommand\cN{{\mathcal N}}

\usepackage[mathlines]{lineno}
\newcommand{\ignore}[1]{}

\title{Generalized Tur\'an problems for small graphs}
\author{D\'aniel Gerbner\footnote{Alfr\'ed R\'enyi Institute of Mathematics, E-mail: \texttt{gerbner.daniel@renyi.hu.} Research supported by the
    National Research, Development and Innovation Office -- NKFIH under the
    grants FK 132060, KKP-133819, KH130371 and SNN 129364.}}

\begin{document}

\maketitle

\begin{abstract}
     For graphs $H$ and $F$, the generalized Tur\'an number $ex(n,H,F)$ is the largest number of copies of $H$ in an $F$-free graph on $n$ vertices. We consider this problem when both $H$ and $F$ have at most four vertices. We give sharp results in almost all cases, and connect the remaining cases to well-known unsolved problems. Our main new contribution is applying the progressive induction method of Simonovits for generalized Tur\'an problems.
\end{abstract}

\section{Introduction}

One of the most studied area of extremal Combinatorics is Tur\'an theory, which seeks to determine $\ex(n,F)$, the largest number of edges in an $F$-free graph on $n$ vertices. A natural generalization is $\ex(n,H,F)$, the largest number of copies of $H$ in $F$-free graphs on $n$ vertices. After several sporadic results (see e.g. \cite{BGY2008,gls,G2012,gypl,HHKNR2013,zykov}), the systematic study of this problem was initiated by Alon and Shikhelman \cite{ALS2016}. Since then, this problem (most commonly referred to as \textit{generalized Tur\'an problem}) has attracted several researchers, see e.g. \cite{cc,cvk,chase,GGMV2017+,GMV2017,GP2017,gp2,gs2017,gstz,mq,wang}. Many bounds and exact results have been proved, for several pairs of graphs. 

In this paper, we examine the case when both $H$ and $F$ have at most four vertices. We collect the known results and prove new results where needed. We feel it is important to put these results in the proper context, thus we state both the existing and the new results in the most general form. We even refer to strong general results when the specific small case we need is trivial. We collect the results concerning small graphs in Table \ref{table} below.

\smallskip

We start with some notation and definition.
We denote by $\cN(H,G)$ the number of copies of $H$ in $G$. The generalized Tur\'an function is $\ex(n,H,F):=\max \{\cN(H,G): G \text{ is an $F$-free graph on $n$ vertices}\}$.

For graphs we use the following notation. $K_n$ is the complete graph on $n$ vertices, $K_{a,b}$ is the complete bipartite graph with parts of size $a$ and $b$, $K_{a,b,c}$ is the complete 3-partite graph with parts of size $a$, $b$ and $c$. The Tur\'an graph $T_r(n)$ is a complete $r$-partite graph on $n$ vertices with each part having size $\lfloor n/r\rfloor$ or $\lceil n/r\rceil$.  
$C_n$ denotes the cycle on $n$ vertices,
$P_n$ denotes the path on $n$ vertices (with $n-1$ edges) and $S_n$ denotes the star on $n$ vertices. $M_\ell$ denotes the matching with $\ell$ edges (thus $2\ell$ vertices). 

We also introduce some less usual notation.
$T_\ell$ is a graph on $l+3$ vertices with $l+3$ edges, that consists of a triangle and $\ell$ other vertices, connected to the same vertex of the triangle ($T_1$ is also called sometimes the paw graph).
$D(k,n)$ is the graph consisting of $\lfloor n/k\rfloor$ copies of $K_k$ and a clique on the remaining vertices. $\overline{K_{s,t}}$ is $K_{s,t}$ with every pair of vertices inside the part of size $s$ connected by an edge. 
 We denote by $G_{n,k,\ell}$ the graph whose vertex set is partitioned into 3 classes, $A$, $B$
and $C$ with $|A| = n-k+\ell$, $|B| = \ell$, $|C| = k - 2\ell$ such that vertices of $B$ have degree $n-1$,  $A$ is an
independent set, $C$ is a clique, and there is no edge between $A$ and $C$. $F(n)$ denotes the friendship graph on $n$ vertices, which has a vertex of degree $n-1$ and a largest matching $M_{\lfloor (n-1)/2\rfloor}$ on the other vertices.

Following \cite{gp2}, if $F$ is a $k$-chromatic graph and $H$ does not contain $F$, then we say that $H$ is \textit{$F$-Tur\'an-good}, if $\ex(n,H,F)=\cN(H,T_{k-1}(n))$ for every $n$ large enough. 
We shorten $K_k$-Tur\'an-good to \textit{$k$-Tur\'an-good}.

\smallskip

Observe that if $F$ contains isolated vertices, then for $n\ge |V(H)|$, the same $n$-vertex graphs contain $F$ and the graph $F'$ we obtain by deleting the isolated vertices from $F$. Therefore, $\ex(n,H,F)=\ex(n,H,F')$. If $H$ contains $k$ isolated vertices, let $H'$ be the graph we obtain by deleting the isolated vertices from $H$. Then each copy of $H'$ in an $n$-vertex graph $G$ extends to a copy of $H$ exactly $\binom{n-|V(H')|}{k}$ ways, thus it is enough to determine $\ex(n,H',F)$. Therefore, we can restrict ourselves to the case neither $F$ nor $H$ contains isolated vertices. With this restriction there are ten graphs on at most four vertices.

We collect a summary of the results in a $10 \times 10$ table. Here we explain what is in the table.
If the column is $F$ and the row is $H$, the entry summarizes what we know about $\ex(n,H,F)$. If the entry is 0, that means $H$ contains $F$, thus $\ex(n,H,F)=0$. Otherwise, the entry does not contain the value of $\ex(n,F)$, it contains a letter and the number of a theorem (or proposition, corollary or observation). The letter E means we know $\ex(n,H,F)$ \textit{exactly}, provided $n$ is large enough. The letter A means we know the asymptotics, 
while the letter $B$ means we only have some bounds and we do not even know the order of magnitude. The numbers after the letter refer to a statement that contains the actual result regarding $\ex(n,H,F)$. Usually it is a more general result.

\begin{table}[h!]\label{table}
\begin{center}
\begin{tabular}{ |p{1cm}||p{1cm}|p{1.3cm}|p{1.3cm}|p{1.3cm}|p{1.3cm}|p{1.3cm}|p{1.3cm}|p{1.3cm}|p{1.3cm}|p{1.3cm}| }
 \hline
  & $K_2$ & $P_3$ & $K_3$ & $M_2$ & $S_4$ & $P_4$ & $C_4$ & $T_1$ & $B_2$ & $K_4$ \\
 \hline
$K_2$  & $0$    & E, \ref{erga}  & E, \ref{Tur} & E, \ref{erga2} & E, \ref{star} & E, \ref{erga} & A, \ref{furedi} & E, \ref{sim} & E, \ref{sim} & E, \ref{Tur}\\
$P_3$ & $0$    & 0  & E, \ref{gpl2} & E, \ref{matching} & E, \ref{regu} & E, \ref{paths} & E, \ref{friends} & E, \ref{tri1} & E, \ref{p3cc} & E, \ref{gpl2} \\
$K_3$ & $0$ & 0 & 0& E, \ref{cons} & E, \ref{gls} & E, \ref{chch} & A, \ref{als} & E, \ref{trian} & B, \ref{als2} &E, \ref{zykov} \\
$M_2$ & $0$  & E, \ref{cseresz}  & E, \ref{gmv} & 0 & E, \ref{regfor} & E, \ref{dani} & A, \ref{matchbip} & E, \ref{matc} & E, \ref{matc} & E, \ref{matc}\\
$S_4$ & $0$  & 0 & E, \ref{induc} & E, \ref{matching} & 0 & E, \ref{starpaths} & E, \ref{korstar} & E, \ref{tegy} & E, \ref{s4b2} & E, \ref{stacli}\\
$P_4$   &$0$ & 0 & E, \ref{gpl2} & 0 & E, \ref{regu} & 0 & A, \ref{coun}  & E, \ref{tri1} & E, \ref{turgoo2} & E, \ref{p4k4} \\
$C_4$ & 0 & 0 & E, \ref{gpl2} & 0& E, \ref{sta} & 0& 0& E, \ref{tri1} & E, \ref{turgoo2} & E, \ref{gpl3}\\
$T_1$ &  $0$ & 0 & 0 &0 &0 &0 & E, \ref{tetel} &0 & B, \ref{rsze} & E, \ref{turgoodgp} \\
$B_2$ & 0 & 0 & 0 & 0 &0 &0 &0 &0 &0 & E, \ref{gpl2} \\
$K_4$ & 0 & 0& 0& 0 & 0&0 &0 &0 &0 & 0\\
\hline
\end{tabular}
\caption{Generalized Tur\'an numbers of small graphs}
\end{center}
\end{table}

The exact results here are proved only for $n$ large enough.
We are not interested in small values of $n$, and we do not mention in the table the cases where in fact we have exact results for all $n$.
In some cases, the result we state follows from a more general theorem, stated only for $n$ large enough, and it would not be hard to obtain the exact value of $\ex(n,H,F)$ for every $n$ in case of the particular small graphs we study here.

Our main new contribution is applying the progressive induction method of Simonovits \cite{sim} for generalized Tur\'an problems and use it to resolve a problem of Gerbner and Palmer \cite{gp2}.

\begin{thm}\label{p3cc}
If $F$ is a 3-chromatic graph with a color-critical edge, then $P_3$ is $F$-Tur\'an-good.
\end{thm}

The rest of this paper is organized as follows. In Section 2 we state the existing results we use. We state them in the most general form, but it is always immediate how they imply the bounds for our specific cases. In Section 3 we state and prove most of our new results. In Section 4 we introduce progressive induction, prove Theorem \ref{p3cc} and another result. We finish the paper with some concluding remarks in Section 5.

\section{Earlier results}

In this section we state earlier results that imply some of the bounds. As the first row of the table corresponds to counting edges, we start with some results concerning ordinary Tur\'an problems. We shall begin with Tur\'an's theorem.

\begin{thm}[Tur\'an, \cite{T1941}]\label{Tur}
We have $\ex(n,K_k)=|E(T_{k-1}(n))|$, i.e. $K_2$ is $k$-Tur\'an-good.
\end{thm}

We say that an edge $e$ of a graph $G$ is \textit{color-critical} if deleting $e$ from $G$ decreases the chromatic number of the graph. Simonovits \cite{sim} showed that the Tur\'an graph has the largest number of edges if we forbid any $k$-chromatic graph with a color-critical edge, provided $n$ is large enough.

\begin{thm}[Simonovits \cite{sim}]\label{sim}
If $F$ has chromatic number $k$ and a critical edge, and $n$ is large enough, then $\ex(n,F)=|E(T_{k-1}(n))|$, i.e. $K_2$ is $F$-Tur\'an-good. 
Moreover, $T_{k-1}(n)$ is the unique extremal graph.
\end{thm}

It is trivial to determine the Tur\'an number of stars. We state it here so that we can refer to it.

\begin{observation}\label{star}
$\ex(n,S_k)=\lfloor (k-2)n/2\rfloor$.
\end{observation}

Erd\H os and Gallai \cite{erga} studied $\ex(n,P_k)$, but obtained the exact value only for $n$ divisible by $k-1$. Faudree and Schelp \cite{fasc} improved their result and showed the following.

\begin{thm}[Faudree and Schelp \cite{fasc} ]\label{erga}
For every $n$ and $k$ we have $\ex(n,P_k)=|E(D(k-1,n))|$.
\end{thm}

\begin{thm}[Erd\H os and Gallai, \cite{erga}]\label{erga2} If $n>2l$, then
$\ex(n,M_\ell)=|E(\overline{K_{k-1,n-k+1}})|$.
\end{thm}

F\"uredi \cite{fur} determined the asymptotics of $\ex(n,K_{2,t})$. For infinitely many values of $n$, the exact value of $\ex(n,C_4)$ was also found by F\"uredi \cite{fure}.

\begin{thm}[F\"uredi, \cite{fur}]\label{furedi}
$\ex(n,K_{2,t})=(1+o(1))\frac{1}{2}\sqrt{t-1}n^{3/2}$.
\end{thm}

Let us continue with results concerning generalized Tur\'an problems. The first such result is due to Zykov \cite{zykov}.

\begin{thm}[Zykov, \cite{zykov}]\label{zykov}
If $r<k$, then $K_r$ is $k$-Tur\'an-good.
\end{thm}

It was generalized by Ma and Qiu \cite{mq} to graphs with a color-critical edge.

\begin{thm}[Ma and Qiu; \cite{mq}]\label{MQ}
Let $F$ be a graph with a color-critical edge and chromatic number more than $r$. Then $K_r$ is $F$-Tur\'an-good.
\end{thm}

Gy\H ori, Pach and Simonovits \cite{gypl} started the study of $k$-Tur\'an-good graphs.

\begin{thm}[Gy\H ori, Pach and Simonovits \cite{gypl}]\label{gpl} 
Let $r\ge 3$, and let $H$ be a $(k-1)$-partite graph with $m>k-1$ vertices,
containing $\lfloor m/(k-1)\rfloor$ vertex disjoint copies of $K_{k-1}$. Suppose
further that for any two vertices $u$ and $v$ in the same component of $H$, there is a sequence $A_1,\dots,A_s$ of $(k-1)$-cliques in $H$ such that $u\in A_1$, $v\in A_s$, and for any $i<s$, $A_i$ and $A_{i+1}$ share $k-2$ vertices.
Then $H$ is $k$-Tur\'an-good.
\end{thm}

\begin{corollary}[Gy\H ori, Pach and Simonovits \cite{gypl}]\label{gpl2} 
Paths and even cycles are $3$-Tur\'an-good and $T_{k-1}(m)$ is $k$-Tur\'an-good.
\end{corollary}

\begin{proposition}[Gy\H ori, Pach and Simonovits \cite{gypl}]\label{gpl2.5} 
If $H$ is a complete multipartite graph, then $\ex(n,H,K_k)=\cN(H,G)$ for some complete $(k-1)$-partite graph $G$.
\end{proposition}

\begin{corollary}[Gy\H ori, Pach and Simonovits; \cite{gypl}]\label{gpl3} 
$C_4$ and $K_{2,3}$ are $k$-Tur\'an-good.
\end{corollary}

A result similar to Theorem \ref{gpl} was obtained by Gerbner and Palmer \cite{gp2}.

\begin{thm}[Gerbner, Palmer \cite{gp2}]\label{turgoodgp}
	Let $H$ be a $k$-Tur\'an-good graph. Let $H'$ be any graph constructed from $H$ in the following way.
	Choose a complete subgraph of $H$ with vertex set $X$, add a vertex-disjoint copy of $K_{k-1}$ to $H$ and join the vertices in $X$ to the vertices of $K_{k-1}$ by edges arbitrarily.
	Then $H'$ is $k$-Tur\'an-good.
\end{thm} 	

As a single vertex is a complete graph, this implies that $T_1$ is also 4-Tur\'an-good.


\begin{proposition}[Gerbner and Palmer \cite{gp2}]\label{turgoo2}
$P_4$ and $C_4$ are $B_2$-Tur\'an-good.
\end{proposition}

Cambie, de Verclos and Kang \cite{cvk} studied the case of forbidden stars.

\begin{proposition}[Cambie, de Verclos and Kang \cite{cvk}]\label{regu} Let $T$ be a tree on $k$ vertices and $n$ be large enough. If $nr$ is even, let $G$ be an arbitrary $(r-1)$-regular graph with diameter more than $k$. If $nr$ is odd, let $G$ be an arbitrary graph with diameter more than $k$ that has $n-1$ vertices of degree $r-1$ and one vertex of degree $r-2$. (Note that $G$ exists because $n$ is large enough.) Then $\ex(n,T,S_r)=\cN(T,G)$.
\end{proposition}

Gy\H ori, Salia, Tompkins and Zamora \cite{gstz} studied the case of forbidden paths.

\begin{thm}[Gy\H ori, Salia, Tompkins and Zamora, \cite{gstz}]\label{paths} 
We have $\ex(n,P_3,P_k)=\cN(P_3,G_{n,k-1, \lfloor (k-2)/2\rfloor})$.
\end{thm}

\begin{thm}[Gy\H ori, Salia, Tompkins and Zamora, \cite{gstz}]\label{starpaths} If $k\ge 4$, $r\ge 3$ and $n$ is large enough, then 
$\ex(n,S_r,P_k)=\cN(S_r,G_{n,k-1, \lfloor (k-2)/2\rfloor})$.
\end{thm}

We remark that in case $k=4$, $G_{n,k-1, \lfloor (k-2)/2\rfloor}=S_n$, thus $\ex(n,S_k,P_4)=\binom{n-1}{k-1}$. Instead of using the above theorem, one could easily deduce this from the fact that every component of a $P_4$-free graph is either a triangle or a star.


\begin{thm}[Chakraborti and Chen \cite{cc}]\label{chch}
For every $n$, $k$ and $r$ we have $\ex(n,K_k,P_r)=\cN(K_k,D(r-1,n))$. 
\end{thm}

Wang \cite{wang} showed the following.


\begin{proposition}[Wang \cite{wang}]\label{cons}
We have \[ex(n,K_k,\ell K_2)=\max\left\{\binom{2\ell-1}{k},\binom{\ell-1}{k}+(n-\ell+1)\binom{\ell-1}{k-1}\right\}.\]
\end{proposition}

The following was shown by Chase \cite{chase}, proving a conjecture of Gan, Loh and Sudakov \cite{gls}.

\begin{thm}[Chase, \cite{chase}]\label{gls}
If $k>2$, then $\ex(n,K_k,S_r)=\cN(K_k,D(r-1,n))$. 
\end{thm}




Alon and Shikhelman \cite{ALS2016} obtained several results. Here we can use the following ones.

\begin{proposition}[Alon and Shikhelman \cite{ALS2016}]\label{als2}
$n^{2-o(1))}=\ex(n,K_3,B_k)=o(n^2)$.
\end{proposition}

\begin{proposition}[Alon and Shikhelman \cite{ALS2016}]\label{als}
$\ex(n,K_3,K_{2,t})=(1+o(1))\frac{1}{6}(t-1)^{3/2}n^{3/2}$.
\end{proposition}

Gerbner and Palmer \cite{GP2017} determined the asymptotic number of paths and cycles of any length in $K_{2,t}$-free graphs.

\begin{proposition}[Gerbner and Palmer \cite{GP2017}]\label{coun}
$\ex(n, P_k, K_{2,t}) = (\frac{1}{2}+ o(1))(t - 1)^{(k-1)/2}n^{(k+1)/2}$.
\end{proposition}

Gerbner, Methuku and Vizer \cite{GMV2017} studied generalized Tur\'an problems when the forbidden graph is disconnected. They also obtained the following result for the case the other graph is disconnected.

\begin{proposition}[Gerbner, Methuku and Vizer \cite{GMV2017}]\label{gmv}
$M_l$ is 3-Tur\'an-good.
\end{proposition}

The \textit{inducibility} of a graph $H$ is the largest number of induced copies of $H$ that an $n$-vertex graph can contain.
Brown and Sidorenko \cite{brosid} showed that for $H=S_4$, the most copies of $H$ are contained in either $K_{k,n-k}$ or $K_{k+1,n-k-1}$, where $k=\lfloor \frac{n}{2}-\sqrt{(3n-4)/2}\rfloor$. As in a triangle-free graph (or a $T_1$-free graph) every copy of a star is induced, this implies the same upper bound for $\ex(n,S_4,K_3)$. As the constructions are triangle-free, this implies the following (for more on the connection of inducibility and generalized Tur\'an problems, see \cite{ggmv}).

\begin{corollary}\label{induc}
$\ex(n,S_4,K_3)=\max\{\cN(S_4,K_{k,n-k}),\cN(S_4,K_{k+1,n-k-1})\}$, where $k=\lfloor \frac{n}{2}-\sqrt{(3n-4)/2}\rfloor$.
\end{corollary}

\section{New results}

In this section we present our new results. We often state them in a more general form than needed. 

\begin{proposition}\label{friends}
\begin{displaymath}
\ex(n,P_3, C_4)=\cN(P_3,F(n))=
\left\{ \begin{array}{l l}
\binom{n}{2} & \textrm{if\/ $n$ is odd},\\
\binom{n}{2}-1 & \textrm{if\/ $n$ is even}.\\
\end{array}
\right.
\end{displaymath}
\end{proposition}

\begin{proof} The lower bounds are given by the friendship graph $F(n)$. Recall that it has a vertex
$v$ of degree $n-1$, and a matching of $\lfloor (n-1)/2\rfloor$ edges on the remaining vertices. Then for any two vertices different from $v$, they are endpoints of a $P_3$ with $v$ in the middle. For $v$ and another vertex $u$, if $u$ is connected to $u'$ in the matching, then $uu'v$ is a $P_3$ with $u$ and $v$ as endpoints. Thus every pair of vertices, except $\{v,w\}$ forms the endpoints of a $P_3$, where $w$ is the vertex not in the matching in case $n$ is even.

For the upper bound, let $G$ be a $C_4$-free graph. We count the copies of $P_3$ by their endpoints; obviously any two vertices have at most one common neighbor by the $C_4$-free property, thus $\cN(P_3,G)\le\binom{n}{2}$.

Let $n$ be even, $G$ be a $C_4$-free graph on $n$ vertices and assume indirectly that $\cN(P_3,G)=\binom{n}{2}$, i.e. every pair of vertices has a common neighbor. Let $v$ be an arbitrary vertex and $U$ be its neighborhood. Observe that any vertex of $U$ has a common neighbor with $v$ only if there is a perfect matching in $U$, thus $|U|$ is even. Also, there cannot be any other edges inside $U$ because of the $C_4$-free property.

Let $U'$ be the set of $n-|U|-1$ vertices not connected to and different from $v$, thus $|U'|$ is odd. Each vertex of $U'$ is connected to exactly one vertex in $U$; at least one because that is the common neighbor with $v$, and at most one because of the $C_4$-free property. Thus there is an odd number of edges between $U$ and $U'$. As each vertex of $U$ has two neighbors outside $U'$, it means the sum of the degrees of vertices in $U$ is odd. Thus there is a vertex of odd degree in $U$. But we have obtained that an arbitrary vertex of $G$ has to be of even degree, a contradiction. 
\end{proof}

\begin{proposition}\label{korstar}
If $r\ge 4$, then $\ex(n,S_r,C_4)=\binom{n-1}{r-1}$. 
\end{proposition}

\begin{proof}
The lower bound is given by the star $S_n$. For the upper bound, let $G$ be a $C_4$-free graph with maximum degree $\Delta$ and consider two of its vertices $u$ and $v$. Let us consider the copies of $S_r$ where $u$ and $v$ are leaves. They have at most one common neighbor, that has to be a center of the $S_r$, and then we have at most $\binom{\Delta-2}{r-3}$ ways to choose the other leaves. This way we count every copy of $S_r$ $\binom{r-1}{2}$ times, thus $\cN(S_r,G)\le \frac{1}{\binom{r-1}{2}}\binom{n}{2}\binom{\Delta-2}{r-3}$. If $\Delta \le n-3$, this finishes the proof.

If $\Delta=n-1$ and $w$ has degree $n-1$, then no other vertex can have degree more than 2, thus $w$ is the only center of copies of $S_r$ and $\cN(S_r,G)=\binom{n-1}{r-1}$. If $\Delta=n-2$ and $w$ has degree $n-2$, let $x$ be the only vertex not adjacent to $w$. Then the degree of $x$ is at most one, as it has at most one common neighbor $y$ with $w$. Observe that the degree of $y$ is at most 3 and the degree of any other vertex is at most 2, thus $\cN(S_r,G)\le \binom{n-2}{r-1}+1$ (where the $+1$ term appears only if $r=4$), finishing the proof.
\end{proof}

\begin{observation}\label{trian}
$\ex(n,K_3,T_1)=\ex(n,K_3,P_4)=\cN(K_3,D(3,n))=\lfloor n/3\rfloor$.
\end{observation}

\begin{proof}
Obviously, in a $T_1$-free or $P_4$-free graph, the vertices of a triangle are not connected to any other vertex, thus the triangles are vertex disjoint.
\end{proof}

The following observations are simple consequences of the facts that an $M_2$-free graph is a star or a triangle and a $P_3$-free graph is a matching.

\begin{observation}\label{matching}
If $k\ge 3$, then $\ex(n,S_k,M_2)=\binom{n-1}{k-1}$. For $k=2$, we have

\begin{displaymath}
\ex(n,S_2,M_2))=
\left\{ \begin{array}{l l}
n & \textrm{if\/ 3 divides $n$},\\
n-1 & \textrm{otherwise}.\\
\end{array}
\right.
\end{displaymath}
\end{observation}

\begin{observation}\label{cseresz}
$\ex(n,M_k,P_3)=\binom{\lfloor n/2\rfloor}{k}$.
\end{observation}

\begin{proposition}\label{dani}
$\ex(n,M_2,P_4)=\cN(M_2,D(3,n))$ if $n\neq 4$ and $\ex(4,M_2,P_4)=1$.

\end{proposition}

\begin{proof}
We prove the statement by induction on $n$, it is trivial if $n\le 4$. Consider $n\ge 5$. Observe that every connected component of a $P_4$-free graph is either a triangle or a star. Let $G$ be a $P_4$-free graph with the maximum number of copies of $M_2$.
Let $G'$ be the graph obtained by removing a star component $S_r$ from $G$ (we are done if there is no such component). Then $\cN(M_2,G)=\cN(M_2,G')+(r-1)|E(G')|$. 

Assume first $G'=D(3,n-r)$. If $r\ge 3$, then we can remove three vertices from $S_r$ and place a triangle on those vertices. It is easy to see that the number of copies of $M_2$ increases this way, a contradiction. If $r=1$ or $r=2$, we are done if $n-r$ is divisible by 3 (as in that case the union of $D(3,n-r)$ and $S_r$ is $D(3,n)$). 

Otherwise, we have an $S_1$ or $S_2$ component in $G'$. We unite the two star components. If they were two isolated vertices, then we add an edge connecting them, if they were an isolated vertex and an edge, we place a triangle there. In these cases the number of copies of $M_2$ clearly increases. If they were two edges, we delete them and place a triangle on three of these vertices. In this case we removed a copy of $M_2$, but increased the number of edges. As there is at least one triangle component in $G$, this increases the number of copies of $M_2$ by at least three, thus the total number of copies of $M_2$ increases, a contradiction.

Assume now $G'\neq D(3,n-r)$.  Note that we can assume $n-r=4$ and $G'=M_2$. Indeed, otherwise both the number of copies of $M_2$ and the number of edges are maximized by $D(3,n-r)$ (using Theorem \ref{erga} and induction). If $r=n-4\ge 3$, just as in the other case above, we can remove three vertices from $S_r$ and place a triangle on those vertices to increase the number of copies of $M_2$, a contradiction. If $r=1$, $G$ consists of two edges and an isolated vertex, but an edge and a triangle contains more copies of $M_2$, a contradiction. If $r=2$, then $G=M_3$, and $2K_3$ contains more copies of $M_2$, a contradiction finishing the proof.
\end{proof}
Using the well-known fact that $\ex(n,F)=O(n)$ only if $F$ is a forest, we can prove an asymptotic result for $\ex(n,M_k,F)$ in case $F$ contains a cycle and we know $\ex(n,F)$ asymptotically.

\begin{observation}\label{matchbip}
If $F$ is not a forest, then $\ex(n,M_k,F)=(1+o(1))\ex(n,F)^k/k!$.
\end{observation}

\begin{proof}
Consider an $F$-free graph. We can pick each of the $k$ edges $\ex(n,F)$ ways, and we count each copy of $M_k$ exactly $k!$ times.

Let us consider now an $F$-free graph $G$ with $\ex(n,F)$ edges. We claim that it contains $(1+o(1))\ex(n,F)^k/k!$ copies of $M_k$. We prove it by induction on $k$. The base case $k=1$ is immediate. Assume that the statement holds for $k-1$ and prove it for $k$. Consider an arbitrary copy of $M_{k-1}$. Then it can be extended to an $M_k$ by any edge not incident to its $2k-2$ vertices. Thus we can choose any of at least $\ex(n,F)-(2k-2)n=(1+o(1))\ex(n,F)$ edges. This way we obtain $\ex(n,F)^k/(k-1)!$, but count each copy of $M_k$ exactly $k$ times.
\end{proof}




\begin{thm}\label{matc}
$M_\ell$ is $F$-Tur\'an-good for every $F$ with a color-critical edge. 
\end{thm}

\begin{proof}
We use induction on $\ell$, the base case $\ell=1$ is Theorem \ref{sim}. Recall that by Observation \ref{matchbip} we have $\ex(n,M_\ell,F)=\Theta(n^{2l})$.
Let $n$ be large enough, $G$ be an $F$-free graph on $n$ vertices with the largest number of copies of $M_\ell$, and let $\chi(F)=k+1$. 

{\bf Case 1.} $G$ has chromatic number more than $k$. We will show that $|E(T_k(n)|\cN(M_{\ell-1},T_k(n-2))-\cN(M_\ell,G)=\Omega(n^{2\ell-1})$ and $|E(T_k(n)|\cN(M_{\ell-1},T_k(n-2))-\cN(M_\ell,T_k(n))=O(n^{2\ell-2})$, which implies that $T_k(n)$ contains more copies of $M_\ell$ than $G$, a contradiction.

A theorem of Erd\H os and Simonovits \cite{valenc} states that if $F$ is $(k+1)$-chromatic and has a color-critical edge, then there is a vertex $v$ of degree at most $(1-\frac{1}{k-4/3})n$ in every $n$-vertex $F$-free graph with chromatic number more than $k$. We claim that $|E(T_k(n)|-|E(G)|=\Omega(n)$. Indeed, by deleting $v$ we obtain a graph with at most $|E(T_k(n-1))|$ edges, and we can delete a vertex from $T_k(n)$ to obtain $T_k(n-1)$. As we delete $\Omega(n)$ more edges in the second case, we are done with the claim.

We count the copies of $M_\ell$ by picking an edge and then picking $M_{\ell-1}$ independently from it. In $G$, this can be done at most $(|E(T_k(n)|-\Omega(n))\cN(M_{\ell-1},T_k(n-2))$ ways. 
Compared to $|E(T_k(n)|\cN(M_{\ell-1},T_k(n-2))$, this is smaller by $\Theta(n^{2\ell-1})$.

We claim that $|E(T_k(n)|\cN(M_{\ell-1},T_k(n-2))-|\cN(M_\ell,T_k(n))|=O(n^{2\ell-2})$, which finishes the proof. In fact we show the stronger statement $|E(T_k(n)||E(T_k(n-2)|\dots|E(T_k(n-2\ell+2)|-|\cN(M_\ell,T_k(n))|=O(n^{2\ell-2})$. Indeed, we can pick the first edge $|E(T_k(n)|$ ways. Then we pick the remaining edges one by one. To pick the $i$th edge, we have to pick an edge from the graph $G_i$ we obtain by deleting the endpoints of the edges picked earlier. $G_i$ is a complete $k$-partite graph on $n-2i+2$ vertices with parts of size at most $\lceil n/k\rceil$ and at least $\lfloor n/k\rfloor-i+1$, as we removed at most $i-1$ vertices from each part. Therefore, we could obtain $T_k(n-2i+2)$ from $G_i$ by moving a constant $c_i$ number of vertices from some parts to other parts. It is easy to see that each such move decreases the number of edges by a constant, therefore we have $|E(G_i)|=|E(T_k(n-2i+2)|-c'_i$ for some constant $c_i'$. Hence $|\cN(M_\ell,T_k(n))|=|E(T_k(n)|(|E(T_k(n-2)|-c'_1)\dots(|E(T_k(n-2\ell+2)|-c_{\ell-1}')$. Each term we subtract from $|E(T_k(n)||E(T_k(n-2)|\dots|E(T_k(n-2\ell+2)|$ has a constant $c_i'$ and at most $\ell-1$ terms that are quadratic, thus the difference is $O(n^{2\ell-2})$.

{\bf Case 2}. $G$ has chromatic number at most $k$. Then we can assume that $G$ is a complete $k$-partite graph, as adding edges do not decrease the number of copies of $M_\ell$ and this way we cannot violate the $F$-free property. We show that making the graph more balanced does not decrease (in fact it increases) the number of copies of $M_\ell$. More precisely,
assume that part $A$ has size $a-1$ and part $B$ has size at least $a+1$, and let $G'$ be $G$ restricted to the other parts. Let us move a vertex $v$ from $B$ to $A$. This means we delete the edges from $v$ to the $a-1$ vertices $u_1,\dots,u_{a-1}$ of $A$, and add edges from $v$ to the other (at least) $a$ vertices $w_1,\dots,w_a$ of $B$. We claim that the resulting graph has more copies of $M_\ell$. We show this by induction on $\ell$, the base case $\ell=1$ is well-known and trivial.

When deleting the edge $vu_i$, we deleted the copies of $M_\ell$ that contained this edge and an $M_{\ell-1}$ on the other vertices. The graph $G_i$ on those other vertices consists of $G'$ and a part of size $a-2$ and a part of size $b\ge a$. Altogether we removed $\sum_{i=1}^{a-1} \cN(M_{\ell-1},G_i)$ copies of $M_\ell$.

When adding the edge $vw_i$, we added copies of $M_\ell$ that contained this edge and an $M_{\ell-1}$ on the other vertices. The graph $G'_i$ on those other vertices consists of $G'$, a part of size $a-1$, and a part of size $b-1\ge a-1$. Altogether we added at least $\sum_{i=1}^{a} \cN(M_{\ell-1},G'_i)$ copies of $M_\ell$. By induction, $G'_i$ has more copies of $M_{\ell-1}$ than $G_i$, finishing the proof (as it shows that we added more copies of $M_\ell$, than what was deleted, even without using the edge $vw_a$). 
\end{proof}

\begin{observation}\label{sta}
If $H$ is $(k-2)$-regular, then $\ex(n,H,S_k)=\lfloor n/|V(H)|\rfloor$.
\end{observation}

\begin{proof} Let $G$ be an $S_k$-free graph. Obviously for any copy of $H$ in $G$, there are no further edges incident to its vertices, thus copies of $H$ are vertex-disjoint. On the other hand, one can take $\lfloor n/|V(H)|\rfloor$ vertex disjoint copies of $H$, and the resulting graph is $S_k$-free.
\end{proof}

\begin{proposition}\label{rsze}
$n^{\ell+2-o(1)}\le \ex(n,T_\ell,B_k)=o(n^{\ell+2})$.
\end{proposition}

\begin{proof}
The upper bound easily follows from Proposition \ref{als2}: there are $o(n^2)$ triangles in a $G$-free graph, and $O(n^\ell)$ ways to choose the $\ell$ additional leaves. For the lower bound, we use the same construction that gives the lower bound in Proposition \ref{als2}. It is a construction by Ruzsa and Szemer\'edi \cite{rsz}, a graph $G$ with $n^{2-o(1)}$ edges where every edge is contained in exactly one triangle. Observe that a vertex with degree $d$ is contained in exactly $d/2$ triangles. 

We have that $G$ contains $n^{2-o(1)}$ triangles. Observe that the number of copies of $T_\ell$ in $G$ is $\sum_{v\in V(G)} \frac{d(v)}{2}\binom{d(v)-2}{\ell}$. Indeed, we pick a vertex $v$, pick a neighbor of $v$ $d_i$ ways, that determines a triangle. We count every triangle
containing $v$ twice. Then we pick $l$ other neighbors of $v$ to be added as leaves. 

By the power mean inequality, we have \[n^{2-o(1)}\le \sum_{v\in V(G)} d(v) \le n\left(\frac{\sum_{v\in V(G)} d(v)^{\ell+1}}{n}\right)^{1/{\ell+1}},\]

which implies $\sum_{v\in V(G)} d(v)^{\ell+1}\ge n^{\ell+2-o(1)}$ and finishes the proof.
\end{proof}




\begin{thm}\label{tetel} 
If $n$ is large enough, then
\begin{displaymath}
\ex(n,T_1, C_4)=\cN(T_1,F(n))=
\left\{ \begin{array}{l l}
\binom{n}{2}-\frac{3(n-1)}{2} & \textrm{if\/ $n$ is odd},\\
\binom{n}{2}-2n-3 & \textrm{if\/ $n$ is even}.\\
\end{array}
\right.
\end{displaymath}

\end{thm}

\begin{proof}
Assume indirectly that there exists an $n$-vertex $C_4$-free graph $G$ with more than $\cN(T_1,F(n))$ copies of $T_1$. We will count the copies of $T_1$ the following way. Consider an unordered pair $\{u,v\}$ of vertices. We count the copies of $T_1$ where one of $u$ and $v$ corresponds to the vertex of degree 1 in $T_1$, and the other corresponds to a vertex of degree two in $T_1$. In $G$, $u$ and $v$ have at most one common neighbor $w$, that has to correspond to the vertex of degree three in $T_1$. Then the last vertex of the $T_1$ is a common neighbor of either $u$ and $w$ or $v$ and $w$. Thus there are at most two copies of $T_1$ obtained this way, and we count every $T_1$ twice this way. We say that these copies of $T_1$ \textit{belong} to the pair $\{u,v\}$, thus every copy of $T_1$ belongs to at most two pairs of vertices. Note that this argument immediately gives the upper bound $\ex(n,T_1, C_4)\le\binom{n}{2}$.

\begin{clm}
There is a vertex of $G$ with degree at least $n-8$.
\end{clm}

\begin{proof}

Let us consider an auxiliary graph $H$ on the same vertex set $V(G)$, where $u$ and $v$ are connected in $H$ if no $T_1$ belongs to them in $G$. Obviously, $H$ has less than $2n-3$ edges by our indirect assumption, thus there is a vertex $x$ with degree at most 3 in $H$. Let $\{x_1,x_2,x_3\}$ contain all the neighbors of $x$ in $H$. Observe that $G$ is a subgraph of $H$. Indeed, if $uv\in E(G)$, and $w$ is their common neighbor, they form a triangle, and  $uw$ and $vw$ both have a common neighbor in the triangle. Thus neither the pair $(u,w)$, nor the pair $(v,w)$ has another common neighbor, that could correspond to the fourth vertex of $T_1$. Thus no copy of $T_1$ belongs to $\{u,v\}$, hence $uv\in E(H)$. This implies that $x$ has degree at most 3 in $G$. 

Assume first that $x$ is connected to $x_1,x_2,x_3$ in $G$. Then for every other vertex $y$, there is a $P_3$ in $G$ from $x$ to $y$, because  they are not connected to $x$ in $H$. Therefore, $y$ is connected to $x_1$, $x_2$ or $x_3$ in $G$, but only one of them, as they have another common neighbor $x$. Let $X_i$ be the set of neighbors of $x_i$ in $G$, that are different from $x,x_1,x_2,x_3$. A vertex in $X_i$ can be connected in $G$ to at most one vertex of $X_1,X_2,X_3$, thus has degree at most 4 in $G$.


A vertex in $X_1$ is connected in $G$ by a $P_3$ to every vertex in $X_1$, but in $X_2$ to at most 
three vertices. Indeed, its only neighbors in $X_1,X_2,X_3$ are each connected to at most one vertex in $X_2$. Therefore  in the auxiliary graph $H$
at least $|X_1|(|X_2|-3+|X_3|-3)$ edges go from $X_1$ to $X_2\cup X_3$. By the same reasoning for $X_2$ and $X_3$, we obtain that 

\begin{eqnarray*}
|E(H)|\ge \frac{|X_1|(|X_2|+|X_3|-6)+|X_2|(|X_1|+|X_3|-6)+|X_3|(|X_2|+|X_1|-6)}{2}=\\
|X_1||X_2|+|X_1||X_3|+|X_2||X_3|-3(|X_1|+|X_2|+|X_3|)=|X_1||X_2|+|X_1||X_3|+|X_2||X_3|-3n+12.
\end{eqnarray*}
In particular, this is greater than $2n-3$ (which is a contradiction) unless the sum of the two smallest set, say $|X_2|+|X_3|$ is at most $5$ (if $n$ is large enough), which implies that $x_1$ has degree at least $n-8$. 

If the degree of $x$ is 2 in $G$, let without loss of generality $x_1$ and $x_2$ be its neighbors, and similarly to the previous case let $X_i$ be the set of neighbors of $x_i$ in $G$ that are different from $x,x_1,x_2$. Then all but at most one of the other vertices ($x_3$) is in $X_1\cup X_2$, as they are connected to $x$ by a $P_3$ in $G$. A vertex in $X_i$ is connected by a $P_3$ in $G$ to every vertex in $X_1$, but at most three vertices in $X_2$ (through its neighbors in $X_1$ and $X_2$, and $z$). Therefore, we have \[2n-3\ge |E(H)|\ge \frac{|X_1|(|X_2-3)+|X_2|(|X_1|-3)}{2}=|X_1||X_2|-3(n-3)/2,\] which implies that either $|X_1|$ or $|X_2|$ is at most 3, hence either $x_1$ or $x_2$ has degree at least $n-6$.

Finally, if $x$ has degree 1 in $G$, its neighbor is connected in $G$ to all but two of the other vertices, thus has degree at least $n-3$.
\end{proof}

Let $u$ have degree at least $n-8$ in $G$. Let $U$ be the set of at most 7 vertices not connected to $u$ and different from $u$. We claim that vertices in $U$ are in at most $7+15\binom{7}{3}=532$ copies of $T_1$. Indeed, each of those vertices is connected to $V(G)\setminus U$ by at most one edge, thus the triangle in $T_1$ is totally inside or totally outside $U$. Let us consider first the triangles totally outside $U$. Every neighbor $v$ of $u$ is in at most one such triangle (that consists of $v$, $u$ and their at most one common neighbor). At most 7 edges go from $U$ to the neighborhood of $u$, and there is only one way any one of those edges can extend a triangle outside $U$ to a copy of $T_1$. Thus there are at most 7 copies of $T_1$ where the triangle is totally outside $U$.

There are at most $\binom{7}{3}$ triangles inside $U$ (obviously there are even fewer, because of the $C_4$-free property). They each have three endpoints, and those points have degree at most 7, thus there are at most 5 ways to extend the triangle to a copy of $T_1$ from that endpoint.

Let us now delete the vertices of $U$ from $G$ to obtain $G'$. On the $n'=n-|U|$ vertices of $G'$, we have a vertex $u$ of degree $n'-1$ in $G'$. Obviously, there can only be a matching on the other vertices of $G'$, thus $G'$ is a subgraph of $F_{n'}$ and $\cN(T_1,G')\le \cN(T_1,F_{n'})$. Therefore, $\cN(T_1,G)\le \cN(T_1,F_{n'})+532<\cN(T_1,F(n))$, a contradiction. For the last inequality, observe that if we add $|U|$ vertices as neighbors of $u$, then each newly added vertex is in $\Omega(n)$ copies of $T_1$. 
\end{proof}

\begin{proposition}\label{stacli}
$S_4$ is 4-Tur\'an-good.
\end{proposition}

\begin{proof} Let $G$ be the $n$-vertex $K_4$-free graph with the most number of copies of $S_4$.
By Proposition \ref{gpl2.5}, we can assume $G=K_{a,b,c}$, we just have to optimize $a,b,c$. The number of $S_4$'s is $a\binom{b+c}{3}+b\binom{a+c}{3}+c\binom{a+b}{3}$. Let us consider a fixed $a$, and choose $b$. The first term is a constant, the other terms are $b\binom{n-b}{3}+(n-a-b)\binom{a+b}{3}$. This is maximized at $b=(n-a)/2$, thus we have that $b$ and $c$ differ by at most one. Similarly $a$ differs from them by at most one, finishing the proof.
\end{proof}

\begin{observation}\label{regfor} If $n\ge 3$, then $\ex(n,M_2,S_4)=n(n-3)/2$.
\end{observation}

\begin{proof}
The lower bound is given by any 2-regular graph, as we can pick an edge, and it has $n-3$ edges independent from it. We count every copy of $M_2$ twice this way.

For the upper bound, observe that an $S_4$-free graph $G$ has at most $n$ edges, and if it has $n$ edges, then it is 2-regular. If $G$ has at most $n-1$ edges, then we can pick an edge at most $n-1$ ways, and another edge at most $n-2$ ways. This gives the upper bound $(n-1)(n-2)/2$, which is one larger than what we claimed. Thus we obtain the desired bound unless above we have equality everywhere, in particular $G$ has $n-1$ edges, and each is independent from all the $n-2$ other edges. But then $G=M_{n-1}$, thus has more than $n$ vertices, a contradiction. 
\end{proof}

\begin{proposition}\label{tegy}
Let $F$ be obtained from $K_r$ by adding a new vertex and connecting it to one of the vertices of the $K_r$. Let $H\neq K_r$ be a connected graph and $n$ be large enough. Then $\ex(n,H,F)=\ex(n,H,K_r)$.
On the other hand, we have $\ex(n,K_r,F)=\cN(K_r,D(r,n))=\lfloor n/r\rfloor$.
\end{proposition}

\begin{proof}
Note first that $\ex(n,H,F)\ge \ex(n,H,K_r)$, as $K_r$ is a subgraph of $F$.

 Let $G$ be an $F$-free graph on $n$ vertices. If there is a $K_r$ in $G$, no other vertex is connected to its vertices. This shows the statement about $\ex(n,K_r,F)$.
 
 Assume first that $H$ has more than $r$ vertices. If there is a $K_r$ in $G$, then its edges cannot be in any copy of $H$. Thus, we can delete all the edges of every $K_r$ from $G$ to obtain a $K_r$-free graph $G'$ with $\cN(H,G')=\cN(H,G)$. As $\cN(H,G')\le \ex(n,H,K_r)$, this finishes the proof.

Assume now $H\neq K_r$ has $p\le r$ vertices, then it has chromatic number at most $r-1$. Therefore, $\cN(H,T_{r-1}(n))=\Omega(n^p)$, hence $\ex(n,H,F)=\Omega(n^p)$. If $p=1$, then the statement is trivial, hence we assume $p>1$ from now on.

Let $G$ be an $F$-free graph on $n$ vertices and assume again that there is a $K_r$ in $G$. Again, no other vertex is connected to its vertices. Let $n$ be large enough in this case. Let $G'$ be the graph we obtain by deleting a copy of $K_r$. We can assume $\cN(H,G')=\ex(n-r,H,F)$, otherwise we could replace $G'$ with an extremal graph to obtain more than $\cN(H,G)$ copies of $H$ on $n$ vertices . 
We have $\cN(H,G)=\cN(H,G')+c$ for a constant $c=\cN(H,K_r)$. 

As $\ex(n,H,F)$ is super-linear and $n-r$ is large enough, there is a vertex $v$ of $G'$ appearing in more than $c$ copies of $H$. Then $v$ is not in any copy of $K_r$ (as in that case its component would be a $K_r$ with only $c$ copies of $H$). Let us add $r$ twins of $v$ to $G'$, i.e. $r$ new vertices connected to exactly the same vertices as $v$. We claim that the resulting graph $G_0$ is $F$-free. Indeed, assume there is an $F$ in $G_0$, and consider the $K_r$ in it, which we denote by $K$. If $K$ does not contain any new vertices, then the additional leaf is a new vertex, but it could be replaced by $v$ to find a copy of $F$ in $G$, a contradiction (recall that $v$ cannot be in $K$). If $K$ contains a new vertex $v'$, then it contains only one new vertex and does not contain $v$, as the new vertices with $v$ form an independent set. But then we could replace $v'$ with $v$ in $K$, to obtain a $K_r$ containing $v$ in $G'$, a contradiction.

Observe that every new vertex $u$ is in more than $c$ copies of $H$ that contains only vertices from $V(G')\setminus\{v\}$ besides $u$. Therefore, $\cN(H,G_0)\ge cr+\cN(H,G')>\cN(H,G)$, a contradiction.
\end{proof}

Using that $P_3$, $P_4$ and $C_4$ are 3-Tur\'an-good by Corollary \ref{gpl2}, we have the following.

\begin{corollary}\label{tri1}
$P_3$, $P_4$ and $C_4$ are $T_1$-Tur\'an-good.
\end{corollary}


\begin{proposition}\label{p4k4}
$P_4$ is $4$-Tur\'an-good.
\end{proposition}

\begin{proof} Let $G$ be a $K_4$-free graph on $n$ vertices. We count the copies of $P_4$ by picking the first and last edge, which are two independent edges. There are at most $\ex(n,M_2,K_4)$ ways to do this, which is $\cN(M_2,T_3(n))$ by Theorem \ref{matc}. 

After picking these two edges, there are five possibilities for the subgraph of $G$ induced on the four vertices of the two edges picked. Either there is a $B_2$ on the four vertices, or a $C_4$, or a $T_1$, or a $P_4$, or an $M_2$. A $B_2$ contains 6 copies of $P_4$ and this way it is counted twice. A $C_4$ contains 4 copies, and is counted twice. A $T_1$ contains 2 copies and is counted once. A $P_4$ contains one copy and is counted once, while an $M_2$ contains no copy and is counted once.

Let $\cN^*(H,F)$ denote the number of induced copies of $H$ in $F$, and let $a=\cN^*(B_2,G)$, $b=\cN^*(C_4,G)$, $c=\cN^*(T_1,G)$, $d=\cN^*(P_4,G)$ and $e=\cN^*(M_2,G)$. Then by the above argument we have $\cN(M_2,G)=2a+2b+c+d+e$, and $\cN(P_4,G)=6a+4b+2c+d$, which implies $\cN(P_4,G)\le 2\cN(M_2,G)+2a$.
Similar equations hold for $T_3(n)$, but no $T_1$, $P_4$ or $M_2$ are induced there, so we have $\cN(M_2,T_3(n))=2\cN^*(B_2,T_3(n))+2\cN^*(C_4,T_3(n))$ and $\cN(P_4,T_3(n))=6\cN^*(B_2,T_3(n))+4\cN^*(C_4,T_3(n))$

Observe that every $B_2$ is induced in a $K_4$-free graph, thus $a=\cN^*(B_2,G)=\cN(B_2,G)\le \ex(n,B_2,K_4)=\cN(B_2,T_3(n))$, where the last equality follows from Corollary \ref{gpl2}. We have $\cN(P_4,G)\le 2\cN(M_2,G)+2a=2\cN(M_2,G)+2\cN(B_2,G)\le 2\cN(M_2,T_3(n))+2\cN(B_2,T_3(n))=6\cN^*(B_2,T_3(n))+4\cN^*(C_4,T_3(n))=\cN(P_4,T_3(n))$.
\end{proof}

\section{Progressive induction}

The progressive induction was introduced by Simonovits \cite{sim}. It is a method to prove statements that hold only for $n$ large enough. In case of ordinary induction, one usually proves the base case easily, as it is on a very small graph, and the induction step is more complicated. However, in case the statement only holds for large $n$, even if the induction step can be proved, the base case might be more complicated. 

This is where progressive induction can be used. Let us describe it informally first. Assume we want to prove that an integer valued quantity $\alpha(G)$ on $n$-vertex graphs takes its maximum on a graph $G_n$ (or on a family of graphs). Ordinary induction assumes that this statement holds for some $n'$, and for larger $n$ it proves that $\alpha$ increases by at most $\alpha(G_n)-\alpha(G_{n'})$. Progressive induction does not have the assumption. In this case one has to prove that $\alpha$ increases by strictly less than $\alpha(G_n)-\alpha(G_{n'})$ (unless the $n$-vertex graph is $G_n$). This means that for small values of $n$, $\alpha(G)$ may be larger on an $n$-vertex graph than $\alpha(G_n)$, but this surplus starts decreasing after a while, and eventually vanishes.

Now we state the key lemma more formally. The actual method works for more than just graphs, but 
for simplicity, we state the lemma only for graphs.

\begin{lemma}[Simonovits \cite{sim}]\label{progi}
Let $\cA\supset \cB$ be families of graphs. Let $f$ be a function on graphs in $\cA$ such that $f(G)$ is a non-negative integer, and if $G$ is in $\cB$, then $f(G)=0$. Assume there is an $n_0$ such that if $n>n_0$ and $G\in\cA$ has $n$ vertices, then either $G\in\cB$, or there exist an $n'$ and a $G'\in\cA$ such that $n/2<n'<n$, $G'$ has $n'$ vertices and $f(G)<f(G')$. Then there exists $n_1$ such that every graph in $\cA$ on more than $n_1$ vertices is in $\cB$.
\end{lemma}

We remark that typically here we want to maximize $\alpha$ on $F$-free graphs, and we conjecture that the extremal graphs belong to a family $\cB_0$. Then $\cA$ is the family of $F$-free graphs that maximize $\alpha$, $\cB=\cA\cap \cB_0$, and $f(G)=\alpha(G)-\alpha(H)$, where $H$ maximizes $\alpha$ in $\cB_0$.

We also use a simple result of Alon and Shikhelman \cite{ALS2016} and the removal lemma.

\begin{proposition}[ Alon and Shikhelman \cite{ALS2016}]\label{as}
We have $\ex(n,H,F)=\Omega(n^{|V(H)|})$ if and only if $F$ is not a subgraph of a blow-up of $H$.
\end{proposition}

\begin{lemma}[Removal lemma] If a graph $G$ contains $o(n^{|V(H)|})$ copies of $H$, then there are $o(n^2)$ edges of $G$, such that deleting them makes the resulting graph $H$-free.

\end{lemma}

We also use a simple extension of Proposition \ref{gpl2.5}. Recall that it states that for a $K_k$-free graph $G$ on $n$ vertices and a complete multipartite graph $H$, there is a complete $(k-1)$-partite $G'$ on $n$ vertices with $\cN(H,G)\le \cN(H,G')$.

\begin{proposition}\label{propi} Let $G$ be a $K_k$-free graph on $n$ vertices, with an independent set $A$ of size $a$, and $H$ be a complete multipartite graph. Then there is a complete $(k-1)$-partite $G'$ on $n$ vertices with $\cN(H,G)\le \cN(H,G')$ such that one of the parts of $G'$ has size at least $a$.
\end{proposition}

\begin{proof}
The proof goes similarly the proof of Proposition \ref{gpl2.5} in \cite{gypl}. We apply the symmetrization process due to Zykov \cite{zykov}. Given two non-adjacent vertices $u$ and $v$ in $G$, we say that we symmetrize $u$ to $v$ if we delete all the edges incident to $u$, and then connect $u$ to the neighbors of $v$. It is well-known that the resulting graph is also $K_k$-free \cite{zykov}, and either symmetrizing $u$ to $v$, or symmetrizing $v$ to $u$ does not decrease the number of copies of $H$ \cite{gypl}, thus we can go through the pairs of non-adjacent vertices and symmetrize one to the other. It is also clear that if symmetrizing does not change anything, then non-adjacent vertices have the same neighborhood, thus $G$ is complete multipartite. To prove Proposition \ref{gpl2.5}, one only has to show that we arrive to such a situation after some symmetrizing, i.e. show that the process terminates after finitely many steps. This is done in \cite{gypl} by showing that either the number of copies of $H$, or the number of pairs with the exact same neighborhood increases.


We will show that by choosing carefully the pairs to symmetrize, we can make sure $A$ is always independent, which will finish the proof.
Let us apply the symmetrization first on pairs with both vertices in $A$. This way after finitely many steps we arrive to a graph $G_1$ where all the vertices in $A$ have the same neighborhood $B$. Then we apply symmetrization anywhere, with the additional condition, that we always symmetrize inside $A$, whenever two vertices of $A$ have different neighborhood. Indeed, it is possible that we symmetrize $u\in A$ to $v\in V\setminus A$, and this way after this step $u$ has a neighborhood that is different from the neighborhood of the other vertices in $A$. However, in this case $v$ is not connected to $u$, thus it is not connected to any vertex of $A$. This way we never add any edge inside $A$.
\end{proof}

\begin{corollary}\label{flenbtwo}
Let $\gamma<1$, $F$ be a 3-chromatic graph with a critical edge, $G$ be an $F$-free graph on $n$ vertices, with an independent set $A$ of size $a<\gamma n$, and $H$ be a complete bipartite graph. Then there is a complete bipartite graph $G'$ on $n$ vertices with $\cN(H,G)\le (1-o(1))\cN(H,G')$ such that one of the parts of $G'$ has size at least $a$.
\end{corollary}

\begin{proof}
$G$ contains $o(n^{3})$ triangles by Proposition \ref{as}, thus we can delete $o(n^2)$ edges to delete all the triangles in $G$ by the removal lemma. This way we removed $o(n^{|V(H)|})$ copies of $H$. Let $G_0$ be the resulting graph. Now we can apply Proposition \ref{propi} to find a complete bipartite graph $G_1$ with at least $\cN(H,G_0)$ copies of $H$, and a part of size at least $a$. Let $G'$ be either $G_1$, or $K_{a,n-a}$, the one with more copies of $H$. Then $ \cN(H,G')=\Omega(n^{|V(H)|}$. Therefore, we have $\cN(H,G)\le \cN(H,G_0)+o(n^{|V(H)|})\le (1-o(1))\cN(H,G')$.
\end{proof}

\begin{lemma}\label{ccb2}
Let $H$ be a bipartite graph and $a_n<n/2$ be integers such that for every $n$ we have $a_n-a_{n-1}\le 1$. Let $G_n=K_{a_n,n-a_n}$ and assume that for every $t$ there is $n_t$ such that for $n>n_t$, $\ex(n,H,B_t)=\cN(H,G_n)$. Then for any 3-chromatic graph $F$ with a color-critical edge, if $n$ is large enough, we have $\ex(n,H,F)=\cN(H,G_n)$.
\end{lemma}

\begin{proof}Observe first that it is enough to prove the statement for $F=K_{s,t}^*$, which denotes $K_{s,t}$ with an edge added inside the part of size $s$. We will use induction on $s$. Note that $K_{2,t}^*=B_t$, thus the base case $s=2$ is the assumption in the statement.

Assume now that $s>2$ and we know that the statement holds for $K_{s-1,t'}^*$ for any $t'$. Let us fix an integer $q$ that is large enough (depending on $s$, $t$ and $H$), and let $G$ be a $K_{s,t}^*$-free graph on $n$ vertices, where $n$ is large enough (depending on $s$, $t$, $q$ and $H$). If $G$ does not contain $K_{s-1,qt}^*$, then it contains at most $\cN(H,G_n)$ copies of $H$ by the induction hypothesis and we are done. Let us assume there is a copy of $K$ of $K_{s-1,qt}^*$ in $G$. Observe that every other vertex $u$ is connected to at most $t-1$ of the vertices in the part of size $qt$ of $K$, otherwise $u$ with its $t$ neighbors in that part and the $s-1$ vertices on the other part would form a $K_{s,t}^*$. 

That means that there are at most $(n-s+1-qt)(s-1+t-1)$ edges from the other vertices to $K$. This implies that there is a vertex $v$ in $K$ that has degree at most $(s+t-2)n/qt$ in $G$. Thus, for 
any $\varepsilon>0$, we can choose a $q$ large enough so that
$v$ is in at most $\varepsilon n^{|V(H)|-1}$ copies of $H$.
Then we apply progressive induction. Let $\cA$ denote the family of extremal graphs for $\ex(n,H,K_{s,t}^*)$, i.e. for every $n$, those $n$-vertex graphs which are $K_{s,t}^*$-free, and contain the most copies of $H$ among such graphs on $n$ vertices. Let $\cB$ denote those elements of $\cA$ that are also $K_3$-free and let $f(G):=\cN(H,G)-\cN(H,G_n)$. Let $n'=n-1$ and $G'$ obtained by deleting $v$ from $G$. Let $G''$ be an $F$-free graph on $n-1$ vertices with $\ex(n-1,H,F)$ copies of $H$, thus $G''\in \cB$. Then $f(G)-f(G'')\le f(G)-f(G')\le\cN(H,G_{n-1})-\cN(H,G_n)+ \varepsilon n^{|V(H)|-1}$. 

To apply Lemma \ref{progi} and finish the proof, we need to show that this number is negative, i.e.
every vertex in $G_n$ is in more than $\varepsilon n^{|V(H)|-1}$ copies of $H$ for some $\varepsilon>0$, finishing the proof (observe that we can obtain $G_{n-1}$ from $G_n$ by deleting a vertex). Indeed, every vertex in the same part of $G_n$ is in the same number of copies of $H$. If they are in $o(n^{|V(H)|}-1)$ copies, then there are $o(n^{|V(H)|})<\ex(n,H,T_2(n))$ copies of $H$ in $G_n$, a contradiction to our assumption that $G_n$ is the extremal graph for $\ex(n,H,F)$.
\end{proof}

Now we are ready to prove Theorem \ref{p3cc}, that we restate here for convenience.

\begin{thm*}
If $F$ is a 3-chromatic graph with a color-critical edge, then $P_3$ is $F$-Tur\'an-good.
\end{thm*}


\begin{proof}
By Lemma \ref{ccb2}, it is enough to prove the statement for $F=B_t$.
 Let $G$ be a $B_t$-free graph on $n$ vertices. First we show that the degrees in $G$ cannot be much larger than $n/2$. Let $c=0.51$ and assume there is a vertex with degree at least $cn$. Observe that every neighbor of $v$ is connected to at most $t-1$ neighbors of $v$. Let $G_0$ be the graph we obtain by deleting all the edges between neighbors of $v$. Then $G_0$ has an independent set of size $cn$. We can apply Corollary \ref{flenbtwo} to show that $G_0$ has at most $(1+o(1))\cN(P_3,K_{cn,(1-c)n})$ copies of $P_3$ (here we also use the fact that making the complete bipartite graph more unbalanced would decrease the number of copies of $P_3$, which follows from a simple calculation). Observe that $G$ has at most $\cN(P_3,G)+O(n^2)$ copies of $P_3$, as the deleted edges all are in $O(n)$ copies of $P_3$. Therefore, $\cN(P_3,G)\le (1+o(1))\cN(P_3,K_{cn,(1-c)n})<\cN(T_2(n))$.

Assume now that $G$ contains a triangle with vertices $u$, $v$ and $w$. Observe that at most $t-2$ other vertices are connected to both $u$ and $v$, and similarly to both $u$ and $w$ or to both $v$ and $w$. Therefore, we have $d(u)+d(v)+d(w)\le n+3t-3$. Let $U$ be the set of the at most $3t-3$ vertices connected to more than one of $u$, $v$ and $w$ (thus $u,v,w\in U$).

 Let $G_1$ be the graph we obtain by deleting $u,v,w$. Let us examine the copies of $P_3$ in $G$. The number of copies containing none of $u,v,w$ is at most $\ex(n-3,P_3,B_t)$. There are 3 copies of $P_3$ inside the triangle.

The other copies of $P_3$ have vertices in both $G_1$ and in the triangle. The number of those copies having their center in 
$V(G_1)\setminus U$ is at most twice the number of edges in $G_1$, as their endpoint has at most one neighbor among $u,v,w$. The number of copies having their center in $U$ and another vertex in the triangle is at most three times the number of edges incident to $U$, thus at most $(9t-9)n$. Finally, the number of copies having their center in the triangle and the other vertices in $G_1$ is $\binom{d(u)-2}{2}+\binom{d(v)-2}{2}+\binom{d(w)-2}{2}$.

Now we will use progressive induction. $\cA$ contains the extremal graphs for $\ex(n,P_3,B_t)$, i.e. for every $n$ the $B_t$-free graphs on $n$ vertices with the most number of copies of $P_3$. $\cB$ consists of those elements of $\cA$ that are $K_3$-free (note that this implies that they are also extremal graphs for $\ex(n,P_3,K_3)$). Let $f(G)=\cN(P_3,G)-\ex(n,P_3,K_3)$. As $\cN(P_3,G)=\ex(n,P_3,B_t)$, we have that $f(G)$ is a non-negative integer, and obviously $f(G)=0$ if $G\in \cB$.

Let $n'=n-3$ and $G'$ be a $B_t$-free graph on $n-3$ vertices with $\ex(n,P_3,B_2)\ge\cN(P_3,G_1)$ copies of $P_3$. Then $f(G)-f(G')$ is at most the number of copies of $P_3$ containing $u$, $v$ or $w$, plus $\ex(n-3,P_3,K_3)-\ex(n,P_3,K_3)$. By the above, the number of copies of $P_3$ containing $u$, $v$ or $w$ is at most 
\begin{equation}\label{eq1}
3+2|E(G')|+(9t-9)n+\binom{d(u)-2}{2}+\binom{d(v)-2}{2}+\binom{d(w)-2}{2}.\end{equation} 

On the other hand, \begin{equation}\label{eq2}
\ex(n,P_3,K_3)-\ex(n-3,P_3,K_3)\ge 3\left(\binom{\lfloor n/2\rfloor}{2}+\lfloor (n-2)^2/4\rfloor-\lceil n/2\rceil\right).\end{equation} 

Indeed, in the Tur\'an graph that is extremal for $\ex(n,P_3,K_3)$, every vertex is in at least $\binom{\lfloor n/2\rfloor}{2}+\lfloor (n-2)^2/4\rfloor$ copies of $P_3$ and for three vertices, we count at most $3\lceil n/2\rceil$ copies of $P_3$ twice. We need to show that (\ref{eq1}) is smaller than (\ref{eq2}). Observe that by Theorem \ref{sim} we have $|E(G')|\le \lfloor (n-3)^2/4\rfloor$, as $G'$ is $B_t$-free and $n$ is large enough. We will show that $\binom{d(u)-2}{2}+\binom{d(v)-2}{2}+\binom{d(w)-2}{2}<3\binom{\lfloor n/2\rfloor}{2}-3\lceil n/2\rceil-3-(9t-9)n$. Recall that each degree is at most $cn$, and $d(u)+d(v)+d(w)\le n+3t-3$. Thus $\binom{d(u)-2}{2}+\binom{d(v)-2}{2}+\binom{d(w)-2}{2}$ is maximized when the three degrees are distributed as unbalanced as possible, implying this sum is at most $2\binom{cn}{2}$, which is smaller than $<3\binom{\lfloor n/2\rfloor}{2}-3\lceil n/2\rceil-3-(9t-9)n$ if $n$ is large enough. This completes the proof.
\end{proof}





It is likely that the above proof can be slightly modified to show $\ex(n,H,B_t)=\ex(n,H,K_3)$ for many other bipartite graphs $H$ in place of $P_3$. I believe it should hold for every complete bipartite graph $H=K_{a,b}$. However, in this case $\ex(n,H,K_3)=\cN(H,K_{m,n-m})$, where $n$ and $m$ might be far apart. When one counts the copies of $K_{a,b}$ having a vertex in the triangle $uvw$, one needs to count the copies of $K_{a-1,b}$ in $G_1$. But, if we use the bound $\cN(K_{a-1,b},G_1)\le \ex(n-3,K_{a-1,b},B_t)$, as in the above proof, we need to deal with the problem, that the $B_t$-free graph with the most number of copies of $H$ might be a complete bipartite graph where the ratio of the parts is far from $m/(n-m)$. This makes the calculations much more complicated. Here we do not attempt to prove a general statement, but we need to deal with $\ex(n,S_4,B_2)$. The following result, combined with Corollary \ref{induc} gives an exact result.

\begin{proposition}\label{s4b2}
If $F$ is 3-chromatic with a color-critical edge, then $\ex(n,S_4,F)=\ex(n,S_4,K_3)$.
\end{proposition}

We only give a sketch, and point out the differences to the proof of Theorem \ref{p3cc}.

\begin{proof}
First observe that it is enough to deal with the case $F=B_t$. Indeed, if $\ex(n,S_4,B_t)=\ex(n,S_4,K_3)$, then there is a complete bipartite extremal graph by Proposition \ref{gpl2.5}, and then Lemma \ref{ccb2} finishes the proof.

By Corollary \ref{induc}, the complete bipartite graph with $\ex(n,S_4,K_3)$ copies of $S_4$ has two parts of size $(\frac{1}{2}+o(1))n$. Therefore, as in the proof of Theorem \ref{p3cc}, we can obtain that every degree is at most $cn$, for $c=0.51$. Again, we pick a triangle with vertices $u,v,w$ and obtain $G_1$ by deleting them. There is a set $U$ of at most $3t-3$ vertices connected to more than one of $u$, $v$ and $w$. The number of copies of $S_4$ is at most $\ex(n-3,S_4,B_t)$ in $G_1$ and at most $\cN(P_3,G_1)+\binom{d(u)-2}{3}+\binom{d(v)-2}{3}+\binom{d(w)-2}{3}+O(n^2)$ additionally, where the $O(n^2)$ term contains those copies that have at least two vertices in $U\cup\{u,v,w\}$. Observe that $\cN(P_3,G_1)\le \ex(n-3,P_3,B_t)=n^3/8+o(n^3)$ and $\binom{d(u)-2}{3}+\binom{d(v)-2}{3}+\binom{d(w)-2}{3}$ is again maximized if they are as unbalanced as possible, thus is at most $2\binom{cn}{3}$. 

We have $\ex(n,S_4,K_3)=\cN(S_4,K_{k,n-k})$ for some $k$ by Corollary \ref{induc} and $\ex(n-3,S_4,K_3)=\cN(S_4,K_{\ell,n-3-\ell})$. It is easy to see that $\ell$ is either $k-1$ or $k-2$, thus
there are three vertices $x,y,z$ of $K_{k,n-k}$ such that deleting them we obtain $K_{\ell,n-3-\ell}$. Hence $\ex(n,S_4,K_3)-\ex(n-3,S_4,K_3)$ is the number of copies of $S_4$ containing $x$, $y$ or $z$. For each of them, there are $3\binom{n/2}{3}+o(n^3)$ copies of $S_4$ where it is the center, and $n^3/16+o(n^3)$ where it is a leaf. There are $o(n^3)$ copies of $S_4$ that are counted multiple times, thus we have $\ex(n,S_4,K_3)-\ex(n-3,S_4,K_3)\ge 3\binom{n/2}{3}+n^3/16+o(n^3)$.
We use progressive induction as in the proof of Theorem \ref{p3cc}. It is again obvious that $f(G)<f(G')$, which finishes the proof.
\end{proof}

\section{Concluding remarks}

$\bullet$ We have studied generalized Tur\'an problems for graphs having at most four vertices. In two cases, we were unable to determine even the order of magnitude of $\ex(n,H,F)$. However, in those cases it would be a major breakthrough in Combinatorics to find the order of magnitude, due to the connection to the Ruzsa-Szemer\'edi theorem.

In some other cases, we could obtain the asymptotics, but not an exact result. One of them is the ordinary Tur\'an problem for $C_4$, which has received a considerable attention, and the exact value of $\ex(n,C_4)$ has been found for infinitely many $n$, as we have mentioned. In case we forbid $C_4$ and count other graphs, we have obtained some exact results, where the friendship graph was the extremal one. This is not the case when counting $K_3$, $M_2$ or $P_4$. Still, one could hope that there is another $C_4$-free graph that has few edges (thus is not considered when dealing with ordinary Tur\'an problems), but many copies of one of the above mentioned graphs. We show that this is not the case.

We claim that if  $G$ is $C_4$-free, then $\cN(P_4,G)\le n|E(G)|/2$ and $\cN(K_3,G)\le |E(G)|/3$. Indeed, let us choose an edge $uv$ and a vertex $w$. There is at most one common neighbor of $u$ and $w$ and another one of $v$ and $w$, and we count every copy of $P_4$ twice this way. Similarly, for an edge $uv$, $u$ and $v$ have at most one common neighbor.
Proposition \ref{coun} shows that $\cN(M_2,G)\le |E(G)|^2$. On the other hand, we have shown $\ex(n,P_4,C_4)=(1+o(1))n\ex(n,C_4)/2$, $\ex(n,K_3,C_4)=(1+o(1))\ex(n,C_4)/3$ and $\ex(n,M_2,C_4)=(1+o(1))\ex(n,C_4)^2$.
Thus in all these cases, for the extremal graph $G$ we have $|E(G)|=(1+o(1))\ex(n,C_4)$. It means determining $\ex(n,K_3,C_4)$, $\ex(n,P_4,C_4)$ or $\ex(n,M_2,C_4)$ exactly is likely as hard as determining $\ex(n,C_4)$. It is possible that one can obtain exact results for infinitely many $n$, using the same ideas as in the ordinary Tur\'an case.

$\bullet$ In each other case we have determined $\ex(n,H,F)$ for $n$ large enough. We did not deal with the case $n$ is small, but probably it is not very hard. Another way to extend these results is to determine all the extremal graphs.

$\bullet$ Another possible direction of future research is to consider graphs on at most five vertices. There are 22 graphs without isolated vertices on five vertices, thus the $10\times 10$ table would be replaced by a $32\times 32$ table, with more than 10 times more entries. Also, all the graphs studied in this paper but $T_1$ belong to at least one well-studied class of graphs, with several results concerning them. There are more exceptions in case of graphs on five vertices, and presumably there are less known results concerning those graphs. 

$\bullet$ It is worth checking what graphs were extremal (or close to extremal) for a given forbidden graph, as they might be also extremal in case we count other graphs. For $K_2$, $P_3$ and $M_2$ there are not many graphs avoiding them. For $K_3$, the extremal graph was always a complete bipartite graph, and it was balanced with one exception. For $S_4$, the extremal graph was sometimes an arbitrary 2-regular graph, but in case of $K_3$ and $C_4$, the extremal graph consisted of vertex-disjoint copies of those graphs (thus it had 2-regular components and potentially some isolated vertices). For $P_4$, the extremal graph was either $S_n$ or $D(3,n)$. For $C_4$, the lower bound was given by either the well-known construction for the ordinary Tur\'an problem concerning $C_4$, or the friendship graph $F(n)$ (in case of counting $S_4$, the lower bound was given by the star $S_n$, which is a subgraph of $F(n)$ and has the same number of copies of $S_4$). In case of $T_1$, the extremal graph was either $D(3,n)$ or a complete bipartite graph, which was balanced with one exception. In case of $B_2$, the lower bound was given by either the construction of Ruzsa and Szemer\'edi, where every edge is in exactly one triangle, or by a complete bipartite graph, which was again balanced with one exception. For $K_4$, in each case the extremal graph was the Tur\'an graph $T_3(n)$.

\end{document}